\documentclass[11pt]{article}
\RequirePackage[OT1]{fontenc}
\RequirePackage[numbers]{natbib}
\usepackage{amsthm,amsmath,amsmath,amssymb,amsfonts,amsthm,graphicx}
\RequirePackage[colorlinks,citecolor=blue,urlcolor=blue]{hyperref}
\usepackage[margin=1.2in]{geometry}
\usepackage{enumerate}
\usepackage[affil-it]{authblk}

\newtheorem{theorem}{Theorem}[section]

\newtheorem{proposition}[theorem]{Proposition}

\newtheorem{lemma}[theorem]{Lemma}

\newtheorem*{Def*}{Definition}

\theoremstyle{definition}

\newtheorem{remark}[theorem]{Remark}

\numberwithin{equation}{section}
\theoremstyle{plain}

\begin{document}
	
\title{Central limit theorems for the real eigenvalues of large Gaussian random matrices}

\author{N. J. Simm}
\affil{Mathematics Institute, University of Warwick, Coventry, CV4 7AL, UK}
\date{}
	\maketitle
	\bigskip
	\abstract{Let $G$ be an $N \times N$ real matrix whose entries are independent identically distributed standard normal random variables $G_{ij} \sim \mathcal{N}(0,1)$. The eigenvalues of such matrices are known to form a two-component system consisting of purely real and complex conjugated points. The purpose of this note is to show that by appropriately adapting the methods of \cite{KPTTZ15}, we can prove a central limit theorem of the following form: if $\lambda_{1},\ldots,\lambda_{N_{\mathbb{R}}}$ are the real eigenvalues of $G$, then for any even polynomial function $P(x)$ and even $N=2n$, we have the convergence in distribution to a normal random variable
	\begin{equation}
	\frac{1}{\sqrt{\mathbb{E}(N_{\mathbb{R}})}}\left(\sum_{j=1}^{N_{\mathbb{R}}}P(\lambda_{j})-\mathbb{E}\sum_{j=1}^{N_{\mathbb{R}}}P(\lambda_{j})\right) \to \mathcal{N}(0,\sigma^{2}(P))
	\end{equation}
	as $n \to \infty$, where $\sigma^{2}(P) = \frac{2-\sqrt{2}}{2}\int_{-1}^{1}P(x)^{2}\,dx$.}
\section{Introduction}
How many eigenvalues of a random matrix are real? This very natural and fundamental question was asked in 1994 by Edelman, Kostlan and Shub \cite{EKS94} who proved that if $G$ is an $N \times N$ matrix of independent identically distributed standard normal variables, and $N_{\mathbb{R}}$ is the number of real eigenvalues of $G$, then
\begin{equation}
\mathbb{E}(N_{\mathbb{R}}) = \sqrt{2N/\pi}+O(1), \qquad N \to \infty \label{edelman}.
\end{equation}
Note that we are not assuming $G$ is symmetric, in the usual parlance we say that $G$ belongs to the so-called Ginibre ensemble of real non-Hermitian random matrices, first considered by Ginibre in 1965 \cite{Gin65}.

In addition to being of instrinsic mathematical interest, the statistics of non-Hermitian matrices also have important applications. The earliest such application is probably due to May \cite{M72} who showed that real random matrices describe the stability properties of large biological systems. Very recently it was shown \cite{FK15} that the counting of the average number of equilibria in a non-linear analogue of May's model can be mapped to the problem of $N_{\mathbb{R}}$ and to the density of real eigenvalues in the Ginibre type ensembles. See also \cite{FD14,MHNSS15} for further applications of $N_{\mathbb{R}}$ to the enumeration of equilibria in complex systems. The question of \textit{fluctuations} in such contexts is usually extremely difficult and has only recently begun to receive attention \cite{S15}. 

The purpose of this article is to describe the asymptotic central limit theorem fluctuations around Edelman and company's estimate \eqref{edelman}. In other words, thinking of \eqref{edelman} as a law of large numbers, what happens when one recenters $N_{\mathbb{R}}$ with respect to its expectation and studies the convergence in law of the fluctuating remainder?

Our approach to this problem is based on a formalism recently developed in \cite{KPTTZ15}, which allowed the authors to characterize the large deviation behaviour for the probability of an anomalously small number of real eigenvalues of $G$. We will show how it is possible to adapt their methods to prove a central limit theorem for the number of real eigenvalues, in addition to the following generalization. From now on let $N=2n$ be even and denote the real eigenvalues of $G$ by $\lambda_{1},\lambda_{2},\ldots,\lambda_{N_{\mathbb{R}}}$. The quantity,
\begin{equation}
X_{n}^{\mathbb{R}}(P) = \sum_{j=1}^{N_{\mathbb{R}}}P(\lambda_{j}/\sqrt{N}), \label{linstat}
\end{equation}
is known as a \textit{linear statistic} (but crucially, note that we only sum the real eigenvalues). The count of real eigenvalues is the special case $X^{\mathbb{R}}_{n}(1) = N_{\mathbb{R}}$.
\begin{theorem}
	\label{th:var}
The variance of the total number of real eigenvalues of the standard $2n \times 2n$ Ginibre real random matrix is given by
\begin{equation}
\mathrm{Var}(N_{\mathbb{R}}) = \frac{2\sqrt{2}}{\sqrt{\pi}}\sum_{k=1}^{n}\frac{\Gamma(2k-3/2)}{\Gamma(2k-1)}-\frac{2}{\pi}\sum_{k_{1}=1}^{n}\sum_{k_{2}=1}^{n}\frac{\Gamma(k_{1}+k_{2}-3/2)^{2}}{\Gamma(2k_{1}-1)\Gamma(2k_{2}-1)} \label{exact}
\end{equation}
and has $n \to \infty$ asymptotics given by
\begin{equation}
\mathrm{Var}(N_{\mathbb{R}}) = (2-\sqrt{2})\mathbb{E}(N_{\mathbb{R}})+O(1), \qquad n \to \infty \label{asympt}
\end{equation}
\end{theorem}
Let us note that the asymptotics \eqref{asympt} also appear in \cite{FN07} and weaker variance estimates (without the constant $2-\sqrt{2}$) were obtained in \cite{TV15} for non-Gaussian matrices. The same asymptotics (including the constant) apply to the generalized eigenvalue problem of real Ginibre matrices \cite{FM12}. Formulae \eqref{exact} and \eqref{asympt} are proved in Section \ref{sec:cov}, including a generalization to the variance of \eqref{linstat} for $P$ an even polynomial, see Proposition \ref{prop:cov}. We also have a central limit theorem for linear statistics:
\begin{theorem}
\label{th:main}
Let $P(x)$ be any even polynomial with real coefficients and let $N=2n$ be even. Then in the limit $n \to \infty$, we have the convergence in distribution
\begin{equation}
\frac{1}{\sqrt{\mathbb{E}(N_{\mathbb{R}})}}(X_{n}^{\mathbb{R}}(P) - \mathbb{E}(X_{n}^{\mathbb{R}}(P)) \to \mathcal{N}(0,\sigma^{2}(P)) \label{normconv}
\end{equation}
where $\mathcal{N}(0,\sigma^{2}(P))$ denotes the normal distribution with mean $0$ and variance
\begin{equation}
\sigma^{2}(P) := \frac{2-\sqrt{2}}{2}\int_{-1}^{1}P(x)^{2}\,dx \label{limvar}
\end{equation}
\end{theorem}

For \textit{Hermitian} random matrices, results of this type continue to occupy a major industry in the field, since at least the 1980s \cite{Jon82} with work continuing unabated to the present day. A quite comprehensive treatment was given by Johansson \cite{Joh98}, who proved that for a general class of Hermitian ensembles, the linear statistic \eqref{linstat} converges as $N \to \infty$, \textit{without normalization}, to a normal random variable with finite variance. The lack of any normalization is usually interpreted as a consequence of strong correlations between the eigenvalues; indeed, for Hermitian matrices, the variance of \eqref{linstat} remains bounded in $N$. In the non-Hermitian case, including all complex eigenvalues in the sum \eqref{linstat} leads again to a bounded variance central limit theorem which is closely related to the Gaussian free field (GFF) \cite{F99,RS06,VR07,AHM11,OR14}, a log-correlated field of great importance in mathematical physics and probability, see \cite{GFF} for a survey. See also \cite{FKS13, LS15, CW15} for further relations between linear statistics of random matrices and log-correlated fields. An important question for future work could be to determine if there a process interpolating between the Poisson fluctuations of Theorem \ref{th:main} and the GFF obtained in \cite{VR07}.

The $N \to \infty$ fate of the sum \eqref{linstat} is therefore quite different to that typically encountered in random matrix theory, requiring a normalization of order $N^{-1/4}$ to ensure distributional convergence. Furthermore, linear statistics of random matrix eigenvalues involving a \textit{random number of terms} have not been studied so widely. However, the Poissonian structure of the limiting Gaussian process can be guessed at in the following way. Viewed as a point process, it is known \cite{TZ11, TKZ12,TZ14} that the unscaled law of the real Ginibre eigenvalues converges as $N \to \infty$ to a system of annihilating Brownian motions taken at time $t=1$. Since the particles move independently, except for annihilation, the terms in the sum \eqref{linstat} are approximately independent. Combined with Edelman's law \eqref{edelman} we may expect that \eqref{linstat} is close to a sum of $O(\sqrt{N})$ independent random variables, for which the classical central limit theorem is applicable. These heuristics are enough to guess \eqref{normconv}, but do not seem to explain the constant\footnote{Interestingly, in the parlance of log-gases, the $2-\sqrt{2}$ prefactor has the physical interpretation as the \textit{compressibility} of the particle system \cite{For13}.}  $2-\sqrt{2}$ in \eqref{limvar}.

For finite $N$, the real spectrum of a Ginibre matrix is not completely independent and therefore \eqref{normconv} requires its own proof. The results of \cite{For13} and \cite{Bee13} indicate that the real eigenvalues have quite interesting statistics, with linear repulsion at close range and Poisson behaviour at large spacings. Specifically, it is shown in \cite{For13} that if $p_{\mathrm{GinOE}}(s)$ is the probability density of real eigenvalue spacings, then 
\begin{equation}
\begin{split}
&p_{\mathrm{GinOE}}(s) \sim c_{0}s, \qquad \qquad s\to 0\\ 
&p_{\mathrm{GinOE}}(s) \sim c_{1}^{2}e^{-c_{1}s}, \qquad \hspace{2pt} s \to \infty \label{mermaid}
\end{split}
\end{equation}
where $c_{0} = 1/(2\sqrt{2\pi})$ and $c_{1} = \zeta(3/2)/c_{0}$. This should be contrasted with the case of random symmetric matrices which have the Wigner-Dyson form (see \cite{Meh04})
\begin{equation}
\begin{split}
&p_{\mathrm{GOE}}(s) \sim (\pi^{2}/6)s, \qquad \hspace{10pt} s \to 0\\ &p_{\mathrm{GOE}}(s) \sim e^{-(\pi s)^{2}/16}, \qquad s \to \infty
\end{split}
\end{equation}
In \cite{Bee13}, the real eigenvalues of non-Hermitian matrices are shown to characterize level crossings in a superconducting quantum dot. Although not of the Ginibre type, the ensembles considered in \cite{Bee13} seem to share the same `mermaid statistics' as \eqref{mermaid}.

Finally, as noted in \cite{FN07}, the real eigenvalues of Ginibre matrices bare a close analogy to the study of real roots of random polynomials of high degree. For a quite general class of random polynomials, variance estimates and central limit theorems for the number of real roots were obtained by Maslova \cite{Masvar,Masclt}. See \cite{TV13} for further references and recent progress in the field of random polynomials. An ensemble of random polynomials closely related to the present study are the $SO(2)$ polynomials defined by $p(x) = \sum_{j=0}^{N}c_{j}x^{j}$ where $c_{j}$ are i.i.d. Gaussian variables with mean zero and variance $\binom{N}{j}$. As for the Ginibre ensemble, the mean and variance of the number of real roots scale as $\sqrt{N}$ \cite{BD97}
\begin{equation}
\mathrm{Var}(N_{\mathbb{R}}^{\mathrm{SO(2)}}) \sim c\sqrt{N}
\end{equation}
where the constant $c=0.57173\ldots$ is close to the Ginibre constant $2-\sqrt{2}=0.5857\ldots$ in \eqref{normconv}. We do not yet have a good explanation for this closeness.

To prove Theorem \ref{th:main}, we rely on the fact that the Ginibre ensemble is a Pfaffian point process. This means that all real and complex correlation functions of the eigenvalues can be written as a Pfaffian \cite{BS09,FN07,FN08,SW08}, in addition to the class of ensemble averages described in \cite{Sin07}. These results rely on the explicit knowledge of the joint probability density function of real and complex eigenvalues \cite{LS91,E93}. In fact, for $f$ even, the moment generating function of the random variable \eqref{linstat} is actually a \textit{determinant} of size $n \times n$. In general, if $f$ is not even it is a Pfaffian of size $2n \times 2n$ that seems more difficult to analyze. From the determinantal formulae, the cumulants of \eqref{linstat} can easily be extracted, and further analysis of their asymptotic behaviour is made possible by appropriately modifying the method used in \cite{KPTTZ15}.
\\\\
\textit{Note}: During the preparation of this article, the arXiv submission \cite{Kop15} appeared, which proves Theorem \ref{th:main} under the different condition that $P$ is compactly supported inside $(-1,1)$. It is likely that combining the methods of \cite{Kop15} and the present article would yield an improved regularity condition on $P$.
	
\section{Proof of the main result}
In the first section we compute the joint cumulant generating function of linear statistics of real and complex eigenvalues. In the second section we calculate the variance and prove Theorem \ref{th:var}. In the final section we bound the higher order cumulants and establish our main result, Theorem \ref{th:main}.
\subsection{Pfaffian and determinantal structures}
The first step towards proving \eqref{normconv} is to calculate the moment generating function of the statistic \eqref{linstat}. A key role (see \cite{KG05} and \cite{Sin07}) is played by the real and complex integrals
\begin{align}
A[h(x)h(y)]_{jk} &= \frac{1}{2}\int_{\mathbb{R}}dx\,\int_{\mathbb{R}}dy\,h(x)h(y)e^{-x^{2}/2-y^{2}/2}P_{j-1}(x)P_{k-1}(y)\mathrm{sign}(y-x) \label{realint}\\
B[g(z)g(\overline{z})]_{jk} &= -2i\int_{\mathbb{C}}g(z)g(\overline{z})P_{j-1}(z)P_{k-1}(\overline{z})\mathrm{sign}(\Im(z))e^{-z^{2}/2-\overline{z}^{2}/2}\mathrm{erfc}(\sqrt{2}|\Im(z)|)\,d^{2}z \label{complexint}
\end{align}
where $\{P_{j}(x)\}_{j \geq 0}$ are a family of degree $j$ monic polynomials. We will choose them to be skew-orthogonal with respect to \eqref{realint} and \eqref{complexint}, as in \cite{FN07} where they were calculated to be
\begin{equation}
P_{2j}(x) = x^{2j}, \qquad P_{2j+1}(x) = x^{2j+1}-2jx^{2j-1} \label{skew-orthog}
\end{equation}
With these polynomials specified, the following skew-orthogonality relation is satisfied:
\begin{equation}
A[1]+B[1] = \mathrm{diag}\bigg\{\begin{pmatrix} 0 & r_{j-1}\\-r_{j-1} & 0 \end{pmatrix}\bigg\}_{j=1}^{n}
\end{equation}
where $r_{j-1} = \sqrt{2\pi}\Gamma(2j-1)$.
	\begin{proposition}
		\label{prop:jointgen}
Let $f \in L^{2}(\mathbb{R})$ and $g \in L^{2}(\mathbb{C})$ be integrable functions and consider the linear statistics
\begin{equation}
X^{\mathbb{R}}_{N}(f) = \sum_{j=1}^{N_{\mathbb{R}}}f(\lambda_{j}), \qquad X^{\mathbb{C}}_{N}(g) = \sum_{j=1}^{N_{\mathbb{C}}}g(z_{j}) \label{realstat}
\end{equation}
Then the joint cumulant generating function of \eqref{realstat} is given by:
	\begin{equation}
	\log\mathbb{E}\left(\mathrm{exp}\left(sX^{\mathbb{R}}_{N}(f)+tX^{\mathbb{C}}_{N}(g)\right)\right) = \frac{1}{2}\log\mathrm{det}\left(I_{2n}+M^{\mathbb{R}}[e^{sf(x)+sf(y)}-1]+M^{\mathbb{C}}[e^{tg(z)+tg(\overline{z})}-1]\right) \label{det}
	\end{equation}
	where $I_{2n}$ is the $2n \times 2n$ identity matrix and $M^{\mathbb{R}/\mathbb{C}}[h(x,y)]$ are $2n \times 2n$ block matrices, where block $(j,k)$ is given by
	\begin{equation}
	\begin{split}
	M^{\mathbb{R}}[h(x,y)]_{jk} &= \frac{1}{\sqrt{2\pi}\Gamma(2j-1)}\begin{pmatrix} -A[h(x,y)]_{2j,2k-1} & -A[h(x,y)]_{2j,2k}\\  A[h(x,y)]_{2j-1,2k-1} & A[h(x,y)]_{2j-1,2k}. \label{mmatrix}
	\end{pmatrix}\\
	M^{\mathbb{C}}[g(z,\overline{z})]_{jk} &= \frac{1}{\sqrt{2\pi}\Gamma(2j-1)}\begin{pmatrix} -B[g(z,\overline{z})]_{2j,2k-1} & -B[g(z,\overline{z})]_{2j,2k}\\ B[g(z,\overline{z})]_{2j-1,2k-1} & B[g(z,\overline{z})]_{2j-1,2k}
\end{pmatrix}
	\end{split}
	\end{equation}
 
	\end{proposition}
	\begin{remark}
The resulting structure of Proposition \ref{prop:jointgen} is reminiscent of formula (3.1) in Tracy and Widom \cite{TW98}, which proved to be extremely useful for the $\beta=2$ Hermitian ensembles.
	\end{remark}
	\begin{proof}
	This follows from a result of Sinclair \cite{Sin07} combined with an important observation of Forrester and Nagao \cite{FN07}. Namely, we apply Theorem 2.1 of \cite{Sin07} but as was observed in \cite{FN07} the proof continues to hold separately for the real and complex eigenvalues. Namely, if we define
	\begin{equation}
	X^{\mathbb{C}}_{N}(g) = \sum_{j=1}^{N_{\mathbb{C}}}g(z_{j})
	\end{equation}
	where $z_{j}$ are the purely complex eigenvalues, then one has a slightly more general statement
	\begin{equation}
	\mathbb{E}(\mathrm{exp}(sX_{N}^{\mathbb{R}}(f)+tX_{N}^{\mathbb{C}}(g))) = \frac{\mathrm{Pf}(A[e^{sf(x)+sf(y)}]+B[e^{tg(z)+tg(\overline{z})}])}{2^{N(N+1)/4}\prod_{j=1}^{N}(\Gamma(j/2))}.
	\end{equation}
	By normalization of the generating function and linearity of the scalar products $A$ and $B$, we have
	\begin{equation}
	\mathbb{E}(\mathrm{exp}(sX_{N}^{\mathbb{R}}(f)+tX_{N}^{\mathbb{C}}(g))) = \frac{\mathrm{Pf}(A[1]+B[1]+A[e^{sf(x)+sf(y)}-1]+B[e^{tg(z)+tg(\overline{z})}-1])}{\mathrm{Pf}(A[1]+B[1])} \label{normgen}
	\end{equation}
    Due to the skew-orthogonality of the $P_{j}'s$, the matrix $A[1]+B[1]$ is block diagonal and skew-symmetric:
	\begin{equation}
	A[1]+B[1] = \mathbf{r} \otimes J,\qquad \mathbf{r} = \mathrm{diag}(r_{0},\ldots,r_{n-1}),\qquad J = \begin{pmatrix} 0 & 1\\ -1 & 0 \end{pmatrix}
	\end{equation}
	with $r_{j} = \sqrt{2\pi}\Gamma(2j+1)$. Taking logarithms and writing the Pfaffians as square roots of determinants gives \eqref{det} after elementary algebra. 
	\end{proof}
	
\begin{remark}
A further simplification occurs whenever the functions $f$ and $g$ are both even. In this case the Pfaffian has a checkerboard structure of zeros and the Pfaffians reduce to determinants of half the size. We then have a \textit{bona fide} determinant
		
\begin{equation}
\mathbb{E}(\mathrm{exp}(sX_{N}^{\mathbb{R}}(f)+tX_{N}^{\mathbb{C}}(g))) = \mathrm{det}\bigg\{\delta_{jk}+\frac{(A[e^{sf(x)+sf(y)-1}]+B[e^{tg(z)+tg(\overline{z})}-1])_{2j-1,2k}}{\sqrt{2\pi\Gamma(2j-1)\Gamma(2k-1)}}\bigg\}_{j,k=1}^{N}  \label{detrem}
\end{equation}
	which is a generalization of formula (6) in
	\cite{KPTTZ15} (setting $g=0$ and $f=1$). See also \cite{FN08} for similar calculations.
		\end{remark}

	To proceed, we will focus our attention on the real eigenvalues and set $g \equiv 0$ from now on. To prove the central limit theorem we will calculate the cumulants of $X^{\mathbb{R}}_{N}(f)$, for which the determinantal formula \eqref{det} is quite well-suited.
	\begin{lemma}
		\label{lem:cumu}
	The $l^{\mathrm{th}}$ order cumulant $\kappa_{l}$ of any even linear statistic $X_{N}^{\mathbb{R}}(f)$ is given by
	\begin{equation}
	\kappa_{l}(f) = l!\sum_{m=1}^{l}\frac{(-1)^{m+1}}{m}\sum_{\substack{\nu_{1}+\ldots+\nu_{m}=l\\\nu_{i} \geq 1}}\frac{\mathrm{Tr}M^{(\nu_{1})}[f]\ldots M^{(\nu_{m})}[f]}{\nu_{1}!\ldots \nu_{m}!} \label{kapp}
	\end{equation}
	where 
	\begin{equation}
	M^{(\nu)}[f]_{jk} := \frac{A[(f(x)+f(y))^{\nu}]_{2j-1,2k}}{\sqrt{2\pi \Gamma(2j-1)\Gamma(2k-1)}}
	\end{equation}
	and $A[f(x,y)]$ is given by \eqref{realint}.
	\end{lemma}
\begin{proof}
From formula \eqref{detrem} with $g=0$, we get
\begin{align}
&[s^{l}]\log\mathbb{E}(\mathrm{exp}(sX_{N}^{\mathbb{R}}(f))) = [s^{l}]\log\mathrm{det}\bigg\{\delta_{jk}+\frac{(A[e^{sf(x)+sf(y)}-1]_{2j-1,2k}}{\sqrt{2\pi\Gamma(2j-1)\Gamma(2k-1)}}\bigg\}_{j,k=1}^{n}\\
&=[s^{l}]\mathrm{Tr}\log\bigg\{\delta_{jk}+\frac{(A[e^{sf(x)+sf(y)}-1]_{2j-1,2k}}{\sqrt{2\pi\Gamma(2j-1)\Gamma(2k-1)}}\bigg\}_{j,k=1}^{n}\\
&=[s^{l}]\sum_{m=1}^{\infty}\frac{(-1)^{m+1}}{m}\mathrm{Tr}\left(\bigg\{\frac{A[e^{sf(x)+sf(y)}-1]_{2j-1,2k}}{\sqrt{2\pi\Gamma(2j-1)\Gamma(2k-1)}}\bigg\}_{j,k=1}^{n}\right)^{m}
\end{align}
Expanding the term $e^{s(f(x)+f(y))}-1$ in a Taylor series and re-ordering the sum gives \eqref{kapp}.
\end{proof}

\subsection{The covariance}
\label{sec:cov}
The main purpose of this section is to prove the following
\begin{proposition}
	\label{prop:cov}
Let $P(x)$ and $Q(x)$ be any even polynomials with real coefficients. Then the covariance of the linear statistics $X^{\mathbb{R}}_{n}[P]$ and $X^{\mathbb{R}}_{n}[Q]$ satisfies the asymptotic formula 
\begin{equation}
	\lim_{n \to \infty}\mathrm{Cov}\bigg\{n^{-1/4}X^{\mathbb{R}}_{n}[P],n^{-1/4}X^{\mathbb{R}}_{n}[Q]\bigg\} = \frac{(2-\sqrt{2})}{\sqrt{\pi}}\int_{-1}^{1}P(x)Q(x)\,dx
\end{equation}
\end{proposition}
To compute the covariance of a general polynomial linear statistic, it suffices to just consider the case of monomials
\begin{equation}
C_{p,q} := \mathrm{Cov}(X^{\mathbb{R}}_{n}(\lambda^{p}),X^{\mathbb{R}}_{n}(\lambda^{q})) \label{covmon}
\end{equation}
Our goal in what follows will be to first find an exact formula for $C_{p,q}$ in Lemmas \ref{lem:cpq} and \ref{lem:exact}, and then compute the large-$n$ asymptotics, which is done in Proposition \ref{prop:dblsum}. Throughout the paper we will make use of the notation
\begin{equation}
f^{(r,s)}_{j,k} := A[x^{r}y^{s}]_{jk}. \label{fpqdef}
\end{equation}
\begin{lemma}
	\label{lem:cpq}
The covariance of two even monomial linear statistics is given for any even matrix dimension $N=2n$ by the formula:
\begin{equation}
\begin{split}
C_{p,q} &= n^{-(p+q)/2)}\sum_{k_{1}=1}^{n}\frac{f^{(p,q)}_{2k_{1}-2,2k_{1}-1}+f^{(q,p)}_{2k_{1}-2,2k_{1}-1}+f^{(0,p+q)}_{2k_{1}-2,2k_{1}-1}+f^{(p+q,0)}_{2k_{1}-2,2k_{1}-1}}{\sqrt{2\pi}\Gamma(2k_{1}-1)}\\
&-n^{-(p+q)/2}\sum_{k_{1},k_{2}=1}^{n}\frac{f^{(0,p)}_{2k_{1}-2,2k_{2}-1}f^{(0,q)}_{2k_{2}-2,2k_{1}-1}+f^{(p,0)}_{2k_{1}-2,2k_{2}-1}f^{(q,0)}_{2k_{2}-2,2k_{1}-1}}{2\pi\Gamma(2k_{1}-1)\Gamma(2k_{2}-1)} \label{cpq}
\end{split}
\end{equation}
\end{lemma}
\begin{proof}
This follows from expressing $C_{p,q}$ in terms of variances using the identity
\begin{equation}
2C_{p,q} = \kappa_{2}(\lambda^{p}+\lambda^{q})-\kappa_{2}(\lambda^{p})-\kappa_{2}(\lambda^{q}) \label{varident}
\end{equation}
The variance terms are then calculated from Lemma \ref{lem:cumu} with $l=2$.
\end{proof}
The coefficients in \eqref{fpqdef} appearing in \eqref{cpq} can be evaluated in the following convenient form.
\begin{lemma}
	\label{lem:exact}
For any even $p$ and $q$, the following exact formula holds: 
\begin{equation}
f^{(p,q)}_{2k_{1}-2,2k_{2}-1} = \Gamma(k_{1}+k_{2}+(p+q)/2-3/2)+qE(k_{1}+p/2,k_{2}+q/2-1) \label{fpq}
\end{equation}
where 
\begin{equation}
E(j,k) := (k-1)!2^{k-1}\sum_{i=0}^{k-1}\frac{\Gamma(i+j-1/2)}{2^{i}i!}. 
\end{equation}
The second term is an error term that satisfies the inequality
\begin{equation}
\begin{split}
E(k_{1}+p/2,k_{2}+q/2-1) &\leq c(k_{2}+q/2-2)!2^{k_{2}}\sum_{i=0}^{\infty}\frac{\Gamma(i+k_{1}+p/2-1/2)}{2^{i}i!}\\
&\leq c\sqrt{n}2^{k_{1}+k_{2}}\Gamma(k_{2}+q/2-3/2)\Gamma(k_{1}+p/2-1/2) \label{facbound}
\end{split}
\end{equation}
where $c$ is a constant independent of $k_{1},k_{2}$ and $n$.
\end{lemma}
\begin{remark}
The first term in $\Gamma(k_{1}+k_{2}+(p+q)/2-3/2)$ in \eqref{fpq} is a natural generalization of the case $p=q=0$ found in \cite{KPTTZ15} and will play just as important a role here in determining the $n \to \infty$ asymptotics.
\end{remark}

\begin{proof}
From the identities $P_{2k_{1}-2}(x)x^{p} = P_{r+2k_{1}-2}(x)$ and $P_{2k_{2}-1}(y)y^{q} = P_{q+2k_{2}-1}+qy^{q+2k_{2}-3}$ we have
\begin{equation}
f^{(p,q)}_{2k_{1}-2,2k_{2}-1} = f^{(0,0)}_{2k_{1}+p-2,2k_{2}+q-1} + qf^{(2k_{1}+p-2,2k_{2}+q-3)}_{0,0}
\end{equation}
The proof is completed by verifying the following identities which are a simple integration exercise:
\begin{equation}
\begin{split}
&f^{(0,0)}_{2k_{1}+p-2,2k_{2}+q-1} = \Gamma(k_{1}+k_{2}+(p+q)/2-3/2)\\
&f^{(2k_{1}+p-2,2k_{2}+q-3)}_{0,0} = E(k_{1}+p/2,k_{2}+q/2-1) 
\end{split}
\end{equation}
\end{proof}

\begin{remark}
The key point is that to prove Theorem \ref{th:main}, it will suffice to only consider the contribution from the first term in \eqref{fpq}. This is proved more generally for all cumulants in Proposition \ref{prop:errors}.
\end{remark}
To extract the asymptotics based on just the first term in \eqref{fpq}, we have the following
\begin{proposition}
\label{prop:dblsum}
Consider the sum
\begin{equation}
S_{p,q} := N^{-(p+q+1)/2}\sum_{k_{1},k_{2}=1}^{N}\frac{\Gamma(k_{1}+k_{2}+\frac{q}{2}-3/2)\Gamma(k_{1}+k_{2}+\frac{p}{2}-3/2)}{\Gamma(2k_{1}-1)\Gamma(2k_{2}-1)} \label{spq}
\end{equation}
Then the following limit holds:
\begin{equation}
\lim_{n \to \infty}S_{p,q} = \sqrt{\pi}\frac{2^{(p+q+1)/2}}{p+q+1}
\end{equation}
\end{proposition}

\begin{proof}
Our strategy will be to bound the sum from above and below. An upper bound can be obtained by extending the $k_{2}$ range of summation to $\infty$:
\begin{equation}
\begin{split}
S_{p,q} &\leq n^{-(p+q+1)/2}\sum_{k_{1}=1}^{n}\frac{\Gamma(k_{1}+(p-1)/2)\Gamma(k_{1}+(q-1)/2)}{\Gamma(2k_{1}-1)}\\
&\times {}_2 F_1([k_{1}+(p-1)/2,k_{1}+(q-1)/2],[1/2],1/4) \label{hypergeom}
\end{split}
\end{equation}
where ${}_2 F_1$ is the classical Gauss hypergeometric function. Since the summand is independent of $n$, it suffices to substitute the $k_{1} \to \infty$ asymptotics in \eqref{hypergeom}. Hence we need the asymptotics of the hypergeometric function with fixed argument and large parameters. These were calculated by several authors using the method of steepest descent, see e.g. \cite{P13}. Indeed, the main result in Section $4$ of \cite{P13} and Stirling's formula imply that
\begin{equation}
\begin{split}
&\frac{\Gamma(k_{1}+(\chi_{p}-1)/2)\Gamma(k_{1}+(q-1)/2)}{\Gamma(2k_{1}-1)}{}_2 F_1([k_{1}+(p-1)/2,k_{1}+(q-1)/2],[1/2],1/4)\\
&\sim \sqrt{\pi}(2k_{1})^{(p+q-1)/2} = \sqrt{\pi}(2k_{1})^{(p+q-1)/2}, \qquad k_{1} \to \infty
\end{split}
\end{equation}
Inserting this into the summand of \eqref{hypergeom} shows that 
\begin{equation}
\begin{split}
\lim_{n \to \infty}S_{p,q} &\leq \lim_{n \to \infty}n^{-(p+q+1)/2}\sqrt{\pi}\sum_{k_{1}=1}^{n}(2k_{1})^{(p+q-1)/2}\\
&=\sqrt{\pi}\frac{2^{(p+q+1)/2}}{p+q+1}
\end{split}
\end{equation}

To obtain a lower bound, we will use the techniques of \cite{KPTTZ15}. The main idea is to write the Gamma functions in the numerator of \eqref{spq} as Gaussian integrals. For $a\geq0$ even, we have
\begin{equation}
\Gamma(k_{1}+k_{2}+a/2-3/2) = 2\int_{\mathbb{R}_{+}}x^{a}\,x^{2k_{1}-2}\,x^{2k_{2}-2}\,e^{-x^{2}}\,dx
\end{equation}
Substituting this expression for the numerator in \eqref{spq} and summing over $k_{1}$ and $k_{2}$ leads to an integral representation
\begin{equation}
S_{p,q} = n^{-(p+q+1)/2}4\int_{\mathbb{R}^{2}_{+}}dx_{1}\,dx_{2}\, x_{1}^{p}x_{2}^{q}\cosh_{n-1}(x_{1}x_{2})^{2}e^{-x_{1}^{2}-x_{2}^{2}}
\end{equation}
where we have employed the hyperbolic cosine series $\cosh_{n-1}(x) = \sum_{k=0}^{n-1}\frac{x^{2k}}{(2k)!}$. By Lemma $4$ of \cite{KPTTZ15}, we have the lower bound 
\begin{equation}
\cosh_{n-1}(x_{1}x_{2}n) \geq h_{n}e^{x_{1}x_{2}n}1(x_{1}x_{2}<T_{n}) \label{lowerbnd}
\end{equation}
where $\lim_{n \to \infty}T_{n}=2$ and $\lim_{n \to \infty}h_{n}=1/2$. Changing variables $x_{i} \to \sqrt{n}x_{i}$ for $i=1,2$ in \eqref{spq} and inserting \eqref{lowerbnd}, we get
\begin{equation}
\begin{split}
S_{p,q} &\geq 4\sqrt{n}h_{n}^{2}\int_{\mathbb{R}_{+}^{2}}dx_{1}\,dx_{2}\,x_{1}^{p}x_{2}^{q}1(x_{1}x_{2} < S_{n})e^{-n(x_{1}-x_{2})^{2}}\\
&\geq 4\sqrt{n}h_{n}^{2}\int_{0}^{\sqrt{T_{n}}}\int_{0}^{\sqrt{T_{n}}}dx_{1}\,dx_{2}\,x_{1}^{p}x_{2}^{q}e^{-n(x_{1}-x_{2})^{2}}\\
&=4\sqrt{n}h_{n}^{2}\frac{1}{2}\int_{0}^{\sqrt{T_{n}}}dR\,\int_{-R}^{R}dz\,\left(\left(\frac{R+z}{2}\right)^{p}\left(\frac{R-z}{2}\right)^{q}\right.\\
&\qquad \qquad \qquad \qquad+\left.\left(\frac{2\sqrt{2}-R-z}{2}\right)^{p}\left(\frac{2\sqrt{2}-R+z}{2}\right)^{q}\right)e^{-nz^{2}}\\
&\sim 4\sqrt{n}h_{n}^{2}\frac{1}{2}\int_{0}^{\sqrt{T_{n}}}dR\,\left((R/2)^{p+q}+((2\sqrt{2}-R)/2)^{p+q}\right)\int_{-R}^{R}dz\,e^{-nz^{2}}\\
&\sim \sqrt{\pi}\frac{2^{(p+q+1)/2}}{(p+q+1)}
\end{split}
\end{equation}
where we used that the domain $\{x_{1}x_{2}<T_{n}\}\cap \mathbb{R}^{2}_{+}$ contains the square $[0,\sqrt{T_{n}}]^{2}$. The subsequent estimates follow from integration by parts. 
\end{proof}
To complete the proof of Proposition \ref{prop:cov}, it is enough to observe that the first line of \eqref{cpq} is asymptotic to
\begin{equation}
4\sum_{k_{1}=1}^{n}\frac{\Gamma(2k_{1}+\frac{p+q}{2}-3/2)}{\Gamma(2k_{1}-1)\sqrt{2\pi}} \sim \frac{2\sqrt{2}(2n)^{(p+q+1)/2}}{p+q+1}
\end{equation}
By Proposition \ref{prop:dblsum}, the second line is asymptotic to $\frac{(2n)^{(p+q+1)/2}2}{\sqrt{\pi}(p+q+1)}$. The difference of these two terms divided by the normalizing factor $(2n)^{(p+q+1)/2}$ is equal to $\frac{(\sqrt{2}-1)}{\sqrt{\pi}}\int_{-1}^{1}x^{p+q}\,dx$. The fact that nothing contributes from the second term in \eqref{fpq} is proved for all cumulants in Proposition \ref{prop:errors}.

\subsection{Higher cumulants and Gaussian fluctuations}
Specialising now to the case $f \equiv P$ of an even polynomial, we will prove in this section that the cumulants of \eqref{linstat} with any order $l \geq 3$ are $O(\sqrt{n})$ as $n \to \infty$. Due to the normalization of order $n^{-1/4}$ in \eqref{normconv}, this bound will be sufficient to conclude the central limit theorem and completes the proof of our main result, Theorem \ref{th:main}. 

By Lemma \ref{lem:cumu} it will suffice to prove that the trace in \eqref{kapp} satisfies the bound
\begin{equation}
\mathrm{Tr}M^{(\nu_{1})}[P]\ldots M^{(\nu_{m})}[P] = O(\sqrt{n}), \qquad n \to \infty \label{cumubound}
\end{equation}
If $P$ is an even polynomial, the above trace is a finite linear combination of terms of the form
\begin{equation}
\mathcal{Z}_{n,m} := n^{-\mathcal{M}_{m}}\sum_{k_{1},\ldots,k_{m}}f^{(2r_{1},2s_{1})}_{2k_{1}-2,2k_{2}-1}\ldots f^{(2r_{m-1},2s_{m-1})}_{2k_{m-1}-2,2k_{m}-1}f^{(2r_{m},2s_{m})}_{2k_{m}-2,2k_{1}-1} \label{partitionsum}
\end{equation}
where $\mathcal{M}_{m} = \sum_{i=1}^{m}(r_{i}+s_{i})$ and as before we have $f^{(2r_{i},2s_{i})}_{2k-2,2j-1} = A[x^{2r_{i}}y^{2s_{i}}]_{2k-1,2j}$ which are explicitly evaluated in \eqref{fpq}. This follows by definition of the trace and expanding $P$ in a basis of monomials. We will prove  \eqref{cumubound} with two Propositions. First we estimate the leading term in \eqref{partitionsum} by substituting just the first factor from \eqref{fpq}, denoted $\Gamma^{(2r_{i},2s_{i})}_{2k_{i}-2,2k_{i+1}-1}$. Then in the second Proposition we deal with the error term in \eqref{fpq} using the bound \eqref{facbound}.
\begin{proposition}
Define
\begin{equation} \Gamma^{(2r_{i},2s_{i})}_{2k_{i}-2,2k_{i+1}-1} = \Gamma(k_{i}+k_{i+1}+r_{i}+s_{i}-3/2)
\end{equation} 
for $i=1,\ldots,m$, where $k_{m+1} \equiv k_{1}$ and define exponents
\begin{equation}
\mathcal{M}_{m} = \sum_{i=1}^{m}(r_{i}+s_{i})
\end{equation}
Then the sum
\begin{equation}
\mathcal{Z}^{(0)}_{n,m} := \sum_{k_{1},\ldots,k_{m}}\Gamma^{(2r_{1},2s_{1})}_{2k_{1}-2,2k_{2}-1}\ldots \Gamma^{(2r_{m-1},2s_{m-1})}_{2k_{m-1}-2,2k_{m}-1}\Gamma^{(2r_{m},2s_{m})}_{2k_{m}-2,2k_{1}-1}
\end{equation}
is $O(n^{1/2+\mathcal{M}_{m}})$ as $n \to \infty$.
\end{proposition}

\begin{proof}
As for the covariance calculation, we write the Gamma factors as Gaussian integrals:
\begin{equation}
\Gamma^{(r_{i},s_{i})}_{2k_{i}-2,2k_{i+1}-1} = \int_{\mathbb{R}}dx\,x^{2r_{i}+2s_{i}}\,x^{2k_{i}-2}\,x^{2k_{i+1}-2}\,e^{-x^{2}}
\end{equation}
shows that
\begin{equation}
	Z_{n,m}^{(0)} = n^{-\mathcal{M}_{m}}\int_{\mathbb{R}^{m}}\,\prod_{j=1}^{m}dx_{j}\,x_{j}^{2r_{j}+2s_{j}}\cosh_{n-1}(x_{j}x_{j+1})e^{-x_{j}^{2}}. \label{zngauss}
\end{equation}
We now use the obvious bound $\cosh_{n-1}(x) \leq \cosh(x)$ on every factor \eqref{zngauss} except one, which we write as a contour integral:
\begin{equation}
\cosh_{n-1}(x_{j}x_{1}) = \oint_{\mathcal{C}}\frac{dz}{2\pi i }\frac{z^{-2n+1}}{1-z^{2}}e^{zx_{j}x_{1}}
\end{equation}
where $\mathcal{C}$ is a small loop around $z=0$. Writing the other $\cosh$ factors as exponentials leads to a finite linear combination of terms of the form
\begin{equation}
Z_{n,m}^{(0)} \leq c_{m}\oint_{}\frac{dz}{2\pi i }\frac{z^{-2n+1}}{1-z^{2}}\int_{\mathbb{R}^{m}}\left(\prod_{j=1}^{m}dx_{j}\,x_{j}^{2r_{j}+2s_{j}}\right)\mathrm{exp}\left(-\mathbf{x}^{\mathrm{T}}A(z)\mathbf{x}\right) \label{gaussint}
\end{equation}
where $\mathbf{x} = (x_{1},\ldots,x_{m})$ and 
\begin{equation}
\label{cyctridiag}
A(z) = \begin{pmatrix} 1 & -\alpha_1/2 & 0 & 0 & \ldots & 0 & -z/2\\
-\alpha_1/2 & 1 & -\alpha_2/2 & 0 & 0 & \ldots & 0\\
0 & -\alpha_{2}/2 & 1 & -\alpha_{3}/2 & 0 & \ldots & 0\\
\vdots & \vdots & \vdots & \ldots & \vdots & \vdots & \vdots\\
0 & \ldots & 0 & -\alpha_{m-3}/2 & 1 & -\alpha_{m-2}/2 & 0\\
0 & \ldots & 0 & 0 & -\alpha_{m-2}/2 & 1 & -\alpha_{m-1}/2\\
-z/2 & 0 & \ldots & 0 & 0 & -\alpha_{m-1}/2 & 1
\end{pmatrix}
\end{equation}
where $\alpha_{i} \in \{1,-1\}$. According to Wick's formula, the integral \eqref{gaussint}, which is essentially the moments of a multivariate Gaussian, can be evaluated explicitly in terms of the determinant and inverse of $A(z)$. We have
\begin{equation}
\det(A(z)) = (m-1)2^{-m}(z-A_{m})((m-1)z+(m+1)A_{m}) \label{detz}
\end{equation}
and 
\begin{equation}
\label{invelements}
\sigma_{jk}(z) := (A^{-1}(z))_{jk} = \frac{a_{jk}z^{2}+b_{jk}z+c_{jk}}{\det(A(z))}
\end{equation}
where $A_{m} = \pm1$ and $a_{jk}$, $b_{jk}$ and $c_{jk}$ are constants. The calculation of $\sigma_{jk}$ is given in Lemma \ref{lem:inv}. Now let $\mathcal{P}_{2}$ be the set of all pairings of elements of the set $\{1,2,\ldots,2\mathcal{M}_{m}\}$. Then Wick's formula tells us that
\begin{equation}
\int_{\mathbb{R}^{m}}\left(\prod_{j=1}^{m}dx_{j}\,x_{j}^{2r_{j}+2s_{j}}\right)\mathrm{exp}\left(-\mathbf{x}^{\mathrm{T}}A(z)\mathbf{x}\right) = \mathrm{det}^{-1/2}\{A(z)\}\sum_{\pi \in \mathcal{P}_{2}}\prod_{(r,s) \in \pi}\sigma_{\chi(r),\chi(s)}(z) \label{wick}
\end{equation}
Inserting \eqref{wick} into \eqref{gaussint}, it is apparent that one can set $A_{m}=1$ in \eqref{detz}, as can be seen by changing variables $z \to zA_{m}$. The number of terms in the product in \eqref{wick} is clearly just $\mathcal{M}_{m}$, so that the integral can be bounded by
\begin{equation}
	Z^{(0)}_{n,m} \leq c_{m}\oint_{}\,\frac{z^{-2n+1}}{1-z^{2}}\frac{K(z)}{(z-1)^{\mathcal{M}_{m}+3/2}((m-1)z+(m+1))^{\mathcal{M}_{m}+1/2}}
\end{equation}
for some other constant $c_{m}>0$. Here $K(z)$ is a polynomial of degree $2\mathcal{M}_{m}$ with no dependence on $n$. As in \cite{KPTTZ15}, deforming contours away from $z=0$ and out to $\infty$, encircling the branch cuts at $(1,\infty)$, $(-\infty,-(m+1)/(m-1))$ and the simple pole at $z=-1$ using that $K(z)$ is analytic shows that the leading contribution for large $n$ comes from the branch point singularity at $z=1$. Integrating by parts $\mathcal{M}_{m}+1$ times, the contribution from the integral along the branch cut $(1,\infty)$ is bounded by a constant times
\begin{equation}
\begin{split}
&\frac{(2n)!}{(2n-\mathcal{M}_{m}-1)!}\int_{1}^{\infty}dy\,\frac{y^{-2n}}{(y-1)^{1/2}}\\
&=\frac{(2n)!}{(2n-\mathcal{M}_{m}-1)!}n^{-1/2}\int_{0}^{\infty}du\,\frac{(1+u/n)^{-2n}}{u^{1/2}} = O(n^{1/2+\mathcal{M}_{m}})
\end{split}
\end{equation}
where we changed variables $u=n(y-1)$ and used the fact that the limit $n \to \infty$ of the last integral is finite.
\end{proof}
It remains to show that the error terms in \eqref{facbound} only give rise to sub-leading contributions in the summation \eqref{partitionsum}.

\begin{proposition}
	\label{prop:errors}
Consider the sum $Z_{n,m}$ in \eqref{partitionsum}, the summands of which consist of a product of $m$ factors. Suppose that $1 \leq c \leq m$ factors are replaced with the error bound in \eqref{facbound}, while the remaining factors are replaced with the leading $\Gamma$-factor in \eqref{fpq}. Denoting the resulting sum by $Z^{(c)}_{n,m}$, we have
\begin{equation}
Z^{(c)}_{n,m} = O(n^{\mathcal{M}_{m}}), \qquad n \to \infty.
\end{equation}
\end{proposition}

\begin{proof}
Due to the factorized form of \eqref{facbound}, the sum \eqref{partitionsum} is a product of $c$ terms of the form 

\begin{equation}
E_{v,\sigma} := \sum_{\substack{k_{1},\ldots,k_{v}}}\frac{b^{(2s_{\sigma(1)})}_{2k_{1}-1}\Gamma^{(2s_{\sigma(2)},2r_{\sigma(2)})}_{2k_{1}-2,2k_{2}-1}\ldots \Gamma^{(2s_{\sigma(v)},2r_{\sigma(v)})}_{2k_{v-1}-2,2k_{v}-1}a^{(2r_{\sigma(v+1)})}_{2k_{v}-2}}{\Gamma(2k_{1}-1)\Gamma(2k_{2}-1)\ldots \Gamma(2k_{v+1}-1)}
\end{equation}
for some permutation $\sigma$ (corresponding to a re-labelling of the $k_{i}'s$) and $1 \leq v \leq j$. The boundary terms $a$ and $b$ come directly from the error term \eqref{facbound} and are given by
\begin{align}
a^{(2r_{\sigma(v+1)})}_{2k_{v+1}-2} &= \sqrt{n}2^{k_{v+1}}\Gamma\left(k_{v+1}+r_{\sigma(v+1)}-3/2\right)\\
b^{(2s_{\sigma(1)})}_{2k_{1}-1} &= 2^{k_{1}}\Gamma\left(k_{1}+s_{\sigma(1)}-1/2\right)
\end{align}
The asymptotic behaviour of $E_{v,\sigma}$ as $n \to \infty$ can be estimated according to the programme already outlined for the leading term. We get 

\begin{equation}
|E_{v,\sigma}| \leq  c\sqrt{n}\oint_{}\frac{dz}{2\pi i}\frac{z^{-2n+1}}{1-z^{2}}\,\int_{\mathbb{R}^{v+1}}x_{1}^{2s_{\sigma(1)}}x_{v+1}^{2r_{\sigma(v+1}-2}\prod_{i=2}^{v}x_{i}^{2s_{\sigma(i)}+2r_{\sigma(i)}}\,\mathrm{exp}\left(-\mathbf{x}^{\mathrm{T}}\tilde{A}(z)\mathbf{x}\right)\,dx_{1}\ldots dx_{v+1} \label{gausserror}
\end{equation}
This time $\tilde{A}(z)$ is symmetric and tridiagonal:
\begin{equation}
\tilde{A}(z) = \begin{pmatrix}1 & -\sqrt{2}z/2 & 0 & 0 & \ldots & 0 & 0\\
-\sqrt{2}z/2 & 1 & -\alpha_{1}/2 & 0 & 0 & \ldots & 0\\
0 & -\alpha_{1}/2 & 1 & -\alpha_{2}/2 & 0 & \ldots & 0\\
\vdots & \vdots & \vdots & \ldots & \vdots & \vdots & \vdots\\
0 & \ldots & 0 & -\alpha_{v-2}/2 & 1 & -\alpha_{v-1}/2 & 0\\
0 & \ldots & 0 & 0 & -\alpha_{v-1}/2 & 1 & -\sqrt{2}\alpha_{v}/2\\
0 & 0 & \ldots & 0 & 0 & -\sqrt{2}\alpha_{v}/2 & 1
\end{pmatrix}
\end{equation}
The Gaussian integral \eqref{gausserror} is again evaluated by Wick's theorem. It's easy to show that
\begin{equation}
\begin{split}
\mathrm{det}(\tilde{A}(z)) &= (1-z^{2})2^{1-v}\\
\tilde{A}^{-1}(z)_{jk} &= \frac{a_{jk}+b_{jk}z+c_{jk}z^{2}}{1-z^{2}}
\end{split}
\end{equation}
where $a_{jk}$, $b_{jk}$ and $c_{jk}$ are constants independent of $n$ and $z$. Therefore \eqref{gausserror} gives a finite linear combination of terms of the form 
\begin{equation}
\sqrt{n}\oint_{}\,\frac{dz}{2\pi i}\frac{z^{-2n+1}P(\alpha z)}{(z^{2}-1)^{3/2+\mathcal{E}_{v}}}
\end{equation}
where $\mathcal{E}_{v} = s_{\sigma(1)}+\sum_{i=2}^{v}(s_{\sigma(i)}+r_{\sigma(i)})+r_{\sigma(v+1)}-1$ and $P(\alpha z)$ is an $n$-independent polynomial. The asymptotics are now dominated by the two branch cuts along $(\pm 1,\infty)$ and one easily sees that the integrals along these cuts are both $O(n^{\mathcal{E}_{v}+1})$ as $n \to \infty$. Taking the product over all such factors gives a bound of order $O(n^{\mathcal{M}_{m}})$, which is what we wanted to show.
\end{proof}
\appendix
\renewcommand{\thelemma}{\Alph{section}\arabic{lemma}}
\section{Miscellaneous Lemmas}
In \cite{KPTTZ15} the determinant of the matrix $A(z)$ in \eqref{cyctridiag} was evaluated explicitly. Due to the application of Wick's theorem, we need the inverse too.
\begin{lemma}
	\label{lem:inv}
The inverse of the cyclic tridiagonal matrix $A(z)$ in \eqref{cyctridiag} is given by
\begin{equation}
(A(z)^{-1})_{rs} := \sigma_{rs}(z) = \sigma(0)_{rs} - \frac{p_{rs}(z)}{D_{m}(z)}
\end{equation}
where $D_{m}(z) = \mathrm{det}(A(z))$, $A_{s} = \prod_{j=1}^{s}\alpha_{j}$ and
\begin{equation}
\sigma(0)_{rs} = (-1)^{r+s}\frac{A_{s}}{A_{r}}2r\frac{j-s+1}{j+1}, \qquad r \leq s
\end{equation}
and
\begin{equation}
p_{rs}(z) = 2z^{2}\sigma(0)_{mm}\sigma(0)_{1s}\sigma(0)_{r1}-2z(z\sigma(0)_{m1}-2)\sigma(0)_{rm}\sigma(0)_{1s}
\end{equation}
\end{lemma}
\begin{proof}
The matrix $A(z)$ in \eqref{cyctridiag} is called a cyclic tri-diagonal matrix. Its inverse can be calculated by noting that it is a rank $2$ perturbation of the tridiagonal matrix $A(0)$:
\begin{equation}
A(z) = A(0)+R(z)S^{\mathrm{T}}
\end{equation}
where $R(z)$ is a $j \times 2$ matrix of zeros except the corners $R(z)_{12} = R(z)_{j2} = -z/2$. Similarly $S$ is a $j \times 2$ matrix of zeros except the corners $S_{11}=1$ and $S_{j2}=1$. The inverse now follows from the algebraic identity
\begin{equation}
A(z)^{-1} = A(0)^{-1}-A(0)^{-1}R(z)(I_{2}+S^{\mathrm{T}}A(0)^{-1}R(z))^{-1}S^{\mathrm{T}}A(0)^{-1}
\end{equation}
The important part for us is the $2 \times 2$ matrix
\begin{equation}
\begin{split}
F &:= (I_{2}+S^{\mathrm{T}}A(0)^{-1}R(z))^{-1}\\
&= \frac{1}{\mathcal{D}_{j}(z)}\begin{pmatrix} z(A(0)^{-1})_{j1}-2 & -z(A(0)^{-1})_{11}\\ -z(A(0)^{-1})_{jj} & z(A(0)^{-1})_{j1}-2
\end{pmatrix}
\end{split}
\end{equation}
where $D_{j}(z) = z^{2}((A(0)^{-1})_{11}(A(0)^{-1})_{jj}-(A(0)^{-1})_{j,1}^{2})+4z(A(0)^{-1})_{j,1}-4$. The inverse of the tridiagonal matrix $A(0)$ can be calculated via classical recurrence relations which can be solved explicitly in this case:
\begin{equation}
(A(0))^{-1}_{rs} =: \sigma(0)_{rs} = (-1)^{r+s}\frac{A_{s}}{A_{r}}2r\frac{j-s+1}{j+1}, \qquad r \leq s
\end{equation}
This completes the proof of the Lemma.
\end{proof}

\textbf{Acknowledgements:} I gratefully acknowledge the support of the Leverhulme Trust Early Career Fellowship (ECF-2014-309).
\bibliographystyle{plain}
\bibliography{fbm_project_refs3}

\begin{thebibliography}{10}

\bibitem{AHM11}
Yacin Ameur, H{\aa}kan Hedenmalm, and Nikolai Makarov.
\newblock Fluctuations of eigenvalues of random normal matrices.
\newblock {\em Duke Math. J.}, 159(1):31--81, 2011.

\bibitem{Bee13}
C.W.J. Beenakker, J.M. Edge, J.P. Dahlhaus, D.I. Pikulin, Shuo Mi, and
  M.~Wimmer.
\newblock {W}igner-{P}oisson {S}tatistics of {T}opological {T}ransitions in a
  {J}osephson {J}unction.
\newblock {\em Phys. Rev. Lett.}, 111:037001, 2013.

\bibitem{BD97}
Pavel Bleher and Xiaojun Di.
\newblock Correlations between zeros of a random polynomial.
\newblock {\em J. Statist. Phys.}, 88(1-2):269--305, 1997.

\bibitem{BS09}
A.~Borodin and C.~D. Sinclair.
\newblock The {G}inibre ensemble of real random matrices and its scaling
  limits.
\newblock {\em Comm. Math. Phys.}, 291(1):177--224, 2009.

\bibitem{E93}
A.~Edelman.
\newblock The {P}robability that a {R}andom {R}eal {G}aussian {M}atrix {H}as
  $k$ {R}eal {E}igenvalues, {R}elated {D}istributions, and the {C}ircular
  {L}aw.
\newblock {\em Journal of Multivariate Analysis}, 60:203--232, 1997.

\bibitem{EKS94}
Alan Edelman, Eric Kostlan, and Michael Shub.
\newblock How many eigenvalues of a random matrix are real?
\newblock {\em J. Amer. Math. Soc.}, 7(1):247--267, 1994.

\bibitem{F99}
P.~J. Forrester.
\newblock Fluctuation formula for complex random matrices.
\newblock {\em J. Phys. A}, 32(13):L159--L163, 1999.

\bibitem{For13}
Peter~J. Forrester.
\newblock Diffusion processes and the asymptotic bulk gap probability for the
  real {G}inibre ensemble.
\newblock {\em J. Phys. A}, 48(32):324001, 14, 2015.

\bibitem{FM12}
Peter~J. Forrester and Anthony Mays.
\newblock Pfaffian point process for the {G}aussian real generalised eigenvalue
  problem.
\newblock {\em Probab. Theory Related Fields}, 154(1-2):1--47, 2012.

\bibitem{FN08}
Peter~J. Forrester and Taro Nagao.
\newblock Skew orthogonal polynomials and the partly symmetric real {G}inibre
  ensemble.
\newblock {\em J. Phys. A}, 41(37):375003, 19, 2008.

\bibitem{FN07}
P.J. Forrester and T.~Nagao.
\newblock Eigenvalue {S}tatistics of the {R}eal {G}inibre {E}nsemble.
\newblock {\em Phys. Rev. Lett.}, 99:050603, 2007.

\bibitem{FD14}
Yan~V. Fyodorov and Pierre Le~Doussal.
\newblock Topology trivialization and large deviations for the minimum in the
  simplest random optimization.
\newblock {\em J. Stat. Phys.}, 154(1-2):466--490, 2014.

\bibitem{FK15}
Y.V. Fyodorov and B.A. Khoruzhenko.
\newblock A {N}onlinear {A}nalogue of {M}ay-{W}igner {I}nstability
  {T}ransition.
\newblock eprint = \texttt{1509.05737}, 2015.

\bibitem{FKS13}
Y.V. Fyodorov, B.A. Khoruzhenko, and N.J. Simm.
\newblock Fractional {B}rownian {M}otion with {H}urst index $h=0$ and the
  {G}aussian {U}nitary {E}nsemble.
\newblock eprint = \texttt{1312.0212}, 2013.

\bibitem{Gin65}
Jean Ginibre.
\newblock Statistical ensembles of complex, quaternion, and real matrices.
\newblock {\em J. Mathematical Phys.}, 6:440--449, 1965.

\bibitem{Joh98}
Kurt Johansson.
\newblock On fluctuations of eigenvalues of random {H}ermitian matrices.
\newblock {\em Duke Math. J.}, 91(1):151--204, 1998.

\bibitem{Jon82}
Dag Jonsson.
\newblock Some limit theorems for the eigenvalues of a sample covariance
  matrix.
\newblock {\em J. Multivariate Anal.}, 12(1):1--38, 1982.

\bibitem{KPTTZ15}
E.~Kanzieper, M.~Poplavskyi, C.~Timm, R.~Tribe, and O.~Zaboronski.
\newblock What is the probability that a large random matrix has no real
  eigenvalues?
\newblock eprint = \texttt{arXiv:1503.07926}, 2015.

\bibitem{KG05}
Eugene Kanzieper and Gernot Akemann.
\newblock Statistics of real eigenvalues in {G}inibre's ensemble of random real
  matrices.
\newblock {\em Phys. Rev. Lett.}, 95(23):230201, 4, 2005.

\bibitem{Kop15}
P.~Kopel.
\newblock Linear {S}tatistics of {N}on-{H}ermitian matrices {M}atching the
  {R}eal or {C}omplex {G}inibre {E}nsemble to {F}our {M}oments.
\newblock eprint = \texttt{1510.02987}, 2015.

\bibitem{LS91}
N.~Lehmann and H-J. Sommers.
\newblock Eigenvalue statistics of random real matrices.
\newblock {\em Physical Review Letters}, 67:941--944, 1991.

\bibitem{LS15}
A.~Lodhia and N.J. Simm.
\newblock Mesoscopic linear statistics of {W}igner matrices.
\newblock eprint = \texttt{arXiv:1503.03533}, 2015.

\bibitem{Masclt}
N.~B. Maslova.
\newblock The distribution of the number of real roots of random polynomials.
\newblock {\em Teor. Verojatnost. i Primenen.}, 19:488--500, 1974.

\bibitem{Masvar}
N.~B. Maslova.
\newblock The variance of the number of real roots of random polynomials.
\newblock {\em Teor. Verojatnost. i Primenen.}, 19:36--51, 1974.

\bibitem{M72}
R.M. May.
\newblock Will a {L}arge {C}omplex {S}ystem be {S}table?
\newblock {\em Nature}, 238:413--414, 1972.

\bibitem{MHNSS15}
D.~Mehta, J.D. Hauenstein, M.~Niemerg, N.J. Simm, and D.A. Stariolo.
\newblock Energy landscape of the finite-size mean-field 2-spin spherical model
  and topology trivialization.
\newblock {\em Phys. Rev. E}, 91:022133, 2015.

\bibitem{Meh04}
Madan~Lal Mehta.
\newblock {\em Random matrices}, volume 142 of {\em Pure and Applied
  Mathematics (Amsterdam)}.
\newblock Elsevier/Academic Press, Amsterdam, third edition, 2004.

\bibitem{OR14}
S.~O'Rourke and D.~Renfrew.
\newblock Central limit theorem for linear eigenvalue statistics of elliptic
  random matrices.
\newblock eprint = \texttt{1410.4586}, 2014.

\bibitem{P13}
R.B. Paris.
\newblock Asymptotics of the {G}auss hypergeometric function with large
  parameters, {I}.
\newblock {\em Journal of Classical Analysis}, 2(2):183--203, 2013.

\bibitem{RS06}
B.~Rider and Jack~W. Silverstein.
\newblock Gaussian fluctuations for non-{H}ermitian random matrix ensembles.
\newblock {\em Ann. Probab.}, 34(6):2118--2143, 2006.

\bibitem{VR07}
Brian Rider and B{\'a}lint Vir{\'a}g.
\newblock The noise in the circular law and the {G}aussian free field.
\newblock {\em Int. Math. Res. Not. IMRN}, (2):Art. ID rnm006, 33, 2007.

\bibitem{GFF}
Scott Sheffield.
\newblock Gaussian free fields for mathematicians.
\newblock {\em Probab. Theory Related Fields}, 139(3-4):521--541, 2007.

\bibitem{Sin07}
Christopher~D. Sinclair.
\newblock Averages over {G}inibre's ensemble of random real matrices.
\newblock {\em Int. Math. Res. Not. IMRN}, (5):Art. ID rnm015, 15, 2007.

\bibitem{SW08}
H-J. Sommers and W.~Wieczorek.
\newblock General eigenvalue correlations for the {G}inibre ensemble.
\newblock {\em J. Phys. A: Math. Theor.}, 41(40), 2008.

\bibitem{S15}
E.~Subag.
\newblock The complexity of spherical p-spin models - a second moment approach.
\newblock eprint = \texttt{1504.02251}, 2015.

\bibitem{TV13}
T.~Tao and V.~Vu.
\newblock Local universality of zeroes of random polynomials.
\newblock eprint = \texttt{1307.4357}, 2013.

\bibitem{TV15}
Terence Tao and Van Vu.
\newblock Random matrices: universality of local spectral statistics of
  non-{H}ermitian matrices.
\newblock {\em Ann. Probab.}, 43(2):782--874, 2015.

\bibitem{TW98}
Craig~A. Tracy and Harold Widom.
\newblock Correlation functions, cluster functions, and spacing distributions
  for random matrices.
\newblock {\em J. Statist. Phys.}, 92(5-6):809--835, 1998.

\bibitem{TKZ12}
Roger Tribe, Siu~Kwan Yip, and Oleg Zaboronski.
\newblock One dimensional annihilating and coalescing particle systems as
  extended {P}faffian point processes.
\newblock {\em Electron. Commun. Probab.}, 17:no. 40, 7, 2012.

\bibitem{TZ11}
Roger Tribe and Oleg Zaboronski.
\newblock Pfaffian formulae for one dimensional coalescing and annihilating
  systems.
\newblock {\em Electron. J. Probab.}, 16:no. 76, 2080--2103, 2011.

\bibitem{TZ14}
Roger Tribe and Oleg Zaboronski.
\newblock The {G}inibre evolution in the large-{$N$} limit.
\newblock {\em J. Math. Phys.}, 55(6):063304, 26, 2014.

\bibitem{CW15}
C.~Webb.
\newblock On the logarithm of the characteristic polynomial of the {G}inibre
  ensemble.
\newblock eprint = \texttt{1507.08674}, 2015.

\end{thebibliography}
\end{document}